\documentclass[10pt,a4paper]{amsart}
\usepackage[latin1]{inputenc}
\usepackage{amsmath}
\usepackage{amsfonts}
\usepackage{amssymb}
\usepackage{amsthm}
\usepackage[usenames, dvipsnames]{color}
\usepackage{tikz}
\usepackage[all]{xy}
\usepackage{geometry}
\definecolor{liens}{rgb}{1,0,0}
\usepackage[colorlinks=true, linkcolor=blue, 
hyperfootnotes=true,citecolor=blue,urlcolor=black]{hyperref}

\newtheorem*{thmintro}{Main Theorem}

\newtheorem*{proposition*}{Proposition~\ref{lem:singcases}}
\newtheorem*{lem**}{Lemma~\ref{lem:doublezero}}
\newtheorem*{thm**}{Theorem~\ref{theo:maintheotransc}}

\newtheorem{thm}{Theorem}[section]
\newtheorem{cor}[thm]{Corollary}
\newtheorem{lemma}[thm]{Lemma}
\newtheorem{lem}[thm]{Lemma}
\newtheorem{prop}[thm]{Proposition}

\newtheorem{defi}[thm]{Definition} %this is redundant

\newenvironment{prf}[1]{\trivlist
\item[\hskip \labelsep{\bf #1.\hspace*{.3em}}]}{~\hspace{\fill}~$\square$\endtrivlist}

\newtheorem{rem}[thm]{Remark} %this is redundant

\newtheorem{exa}[thm]{Example}
\newtheorem{assumption}[thm]{Assumption}

\numberwithin{equation}{section}
\def\iup{{\tilde{\iota}}}

\def\Z{\mathbb{Z}}
\def\C{\mathbb{C}}
\def\R{\mathbb{R}}
\def\Q{\mathbb{Q}}

\def\P1{\mathbb{P}^{1}}

\def\beq{\begin{equation}}
\def\eeq{\end{equation}}

\def\Etproj{\overline{E_t}}

\def\P2{\mathbb{P}^{2}}

\def\CX{{\mathbb C}}

\def\P1{\mathbb{P}^{1}}

\def\calM{{\mathcal{M}}}

\def\a{{\alpha}}
\def\b{{\beta}}
\begin{document}
\title{Walks in the quarter plane: genus zero case}
\author{Thomas Dreyfus}
\address{Institut de Recherche Math\'ematique Avanc\'ee, U.M.R. 7501 Universit\'e de Strasbourg et C.N.R.S. 7, rue Ren\'e Descartes 67084 Strasbourg, FRANCE}
\email{dreyfus@math.unistra.fr}
\author{Charlotte Hardouin}
\address{Universit\'e Paul Sabatier - Institut de Math\'ematiques de Toulouse, 118 route de Narbonne, 31062 Toulouse.}
\email{hardouin@math.univ-toulouse.fr}
\author{Julien Roques}
\address{Univ Lyon, Universit\'e Claude Bernard Lyon 1, CNRS UMR 5208, Institut Camille Jordan, 43 blvd. du 11 novembre 1918, F-69622 Villeurbanne cedex, France}
\email{roques@math.univ-lyon1.fr}
\author{Michael F. Singer }
\address{Department of Mathematics, North Carolina State University,
Box 8205, Raleigh, NC 27695-8205, USA}
\email{singer@ncsu.edu}
%%%%%%%%%%%%%%%%%%%%%%%%%%%%%%%%%%%%%%%%%%%%%%


\thanks{This project has received funding from the European Research Council (ERC) under the European Union's Horizon 2020 research and innovation programme under the Grant Agreement No 648132. The second author would like to thank the ANR-11-LABX-0040-CIMI within
the program ANR-11-IDEX-0002-0 for its partial support. The second author's work is also supported by ANR Iso-Galois. 
The work of the third author has been partially supported by the LabEx PERSYVAL-Lab (ANR-11-LABX-0025-01) funded by the French program Investissement d'avenir. 
The work of the fourth author was partially supported by a grant from the Simons Foundation (\#349357, Michael Singer). }

\date{\today}

\bibliographystyle{amsalpha} 

\begin{abstract} 
We use Galois theory of difference  equations  to study the nature of the generating series of (weighted) walks in the quarter plane with genus zero kernel curve. Using this approach, we prove that the generating series do not satisfy any nontrivial (possibly nonlinear) algebraic differential equation with rational coefficients.  
\end{abstract}
\maketitle
\tableofcontents
\pagestyle{myheadings}
\markboth{T.~DREYFUS, C.~HARDOUIN, J.~ROQUES, M.F.~SINGER}{WALKS IN THE QUARTER PLANE: GENUS ZERO CASE}
\sloppy

\pagebreak

\section*{Introduction}  The generating series of lattice walks in the quarter plane have garnered much interest in recent years.  In \cite{DHRS}, we introduced a new method that allowed us to study the nature of the generating series of many lattice walks with small steps ({\it i.e.}, whose step set  is a subset of $\{ -1,0,1\}^2\backslash \{(0,0)\}$) in the quarter plane. In particular, the paper \cite{DHRS} is concerned with the differential nature of these generating series, the basic question being:  which of them satisfy differential equations? The present paper is a continuation of this research.
We will study weighted  models of  walks with small steps in the quarter plane $\Z_{\geq 0}^{2}$. More precisely, let $(d_{i,j})_{(i,j)\in\{0,\pm 1\}^{2}}$ be a family of elements of $\Q\cap [0,1]$ such that $\sum_{i,j} d_{i,j}=1$.  We encode the eight cardinal directions of the plane by pairs of integers $(i,j)$ with $i,j \in \{0,1\}$.  We  consider a weighted walk in the quarter plane $\Z_{\geq 0}^{2}$ satisfying the following properties: 
\begin{itemize} 
\item it starts at $(0,0)$; 

\item it  takes steps in a certain subset of the set of cardinal directions, which is called the \emph{model of the walk}.
\end{itemize}
For $(i,j)\in\{0,\pm 1\}^{2}\backslash\{(0,0)\}$ (resp. $(0,0)$), the  element   $d_{i,j}$ is  a  weight on the step $(i,j)$  and can be viewed as  the probability for the walk  to go in the direction $(i,j)$ (resp.  to stay  at the same position).  The step set or the \emph{ model} of the walk  corresponds  the set of directions with nonzero weights, that is,
$$
\{ (i,j ) \in \{0,\pm 1\}^{2} \backslash\{(0,0)\}  | d_{i,j} \neq 0\}.
$$
If $d_{0,0}=0$ and if the nonzero $d_{i,j}$ all have the same value, we say that  the model  is unweighted. 

{ The {\it weight of the walk} is defined to be the product of the weights of its component steps.} For any $(i,j)\in \Z_{\geq 0}^{2}$ and any $k\in \Z_{\geq 0}$, we let $q_{i,j,k}$ be the {sum of the weights of all walks reaching}  the position $(i,j)$ from the initial position $(0,0)$ after $k$ steps. We introduce the corresponding trivariate generating series\footnote{In several papers it is not assumed that $\sum_{i,j} d_{i,j}=1$. But after a rescaling of the $t$ variable, we may always reduce to the case  $\sum_{i,j} d_{i,j}=1$.}
$$
Q(x,y,t):=\displaystyle \sum_{i,j,k\geq 0}q_{i,j,k}x^{i}y^{j}t^{k}.
$$

The typical questions considered in the literature are: 
\begin{itemize}
\item is $Q(x,y,t)$ algebraic over $\Q(x,y,t)$?  
\item is $Q(x,y,t)$ $x$-holonomic (resp. $y$-holonomic), {\it i.e.}, is $Q(x,y,t)$, seen as a function of $x$, a solution of some nonzero  linear differential equation with coefficients in  $\Q(x,y,t)$?
\item is $Q(x,y,t)$ $x$-differentially algebraic  (resp. $y$-differentially algebraic), {\it i.e.}  is $Q(x,y,t)$, seen as a function of $x$, a solution of some nonzero  (possibly nonlinear) polynomial differential equation with coefficients in  $\Q(x,y,t)$?  In case of a negative answer, we say that $Q(x,y,t)$ is $x$-differentially transcendental  (resp. $y$-differentially transcendental)\footnote{ We changed the terminology we used in \cite{DHRS}, namely hyperalgebraic and hypertranscendent, because we believe that differentially algebraic and differentially transcendental are more transparent terms.}. 
\end{itemize}

Before describing our main result, we will briefly describe the state of the art. In the seminal paper \cite{BMM}, Bousquet-M\'elou and Mishna studied such questions in the {unweighted case} (see also \cite{Mishna09}). Taking symmetries into consideration and eliminating unweighted  models equivalent to  models  on the half plane (whose generating series is algebraic), Bousquet-M\'elou and Mishna first showed that, amongst the 256 possible   unweighted models, 
it is sufficient to study the above questions for an explicit list of 79  unweighted models. 
 Following  ideas of Fayolle, Iasnogorodski and Malyshev   (see for instance \cite{fayolle1999random,FIM}),  they  associated to each  unweighted model 
 a group of birational automorphisms of $\C^{2}$ and classified the unweighted models
 accordingly.  They found that 23 of the 79 above-mentioned { unweighted models
 were associated with a finite group and showed that for all but one of these 23 models, the generating series was $x$-, $y$- and $t$-holonomic; the remaining one was shown to have the same property by Bostan, van Hoeij and Kauers in \cite{BostanKauersTheCompleteGenerating}. In \cite{BMM}, Bousquet-M\'elou and Mishna conjectured that the 56  unweighted models whose associated group is infinite are not holonomic.  Furthermore, following Fayolle, Iasnogorodski and Malyshev, the 56 unweighted models may be gathered into two families according to the genus of an algebraic curve, called the {\it kernel curve}, attached to each model:  
\begin{itemize}
\item 5 of these  unweighted models
lead to a curve of genus zero; they will be called the genus zero   unweighted models
\item 51 of them lead to a curve of genus one; they will be called the genus one  unweighted models
\end{itemize}   
In \cite{KurkRasch}, Kurkova and Raschel showed that the 51 genus one  unweighted models   
with infinite group have nonholonomic generating series 
(see also \cite{BRS,RaschelJEMS}).  Recently, Bernardi, Bousquet-M\'elou and Raschel \cite{BBMR15, BBMR17} have shown that 9 of these 51 unweighted models have $x$- and $y$-differentially algebraic generating series, despite the fact that they are not $x$- or $y$-holonomic.

In \cite{DHRS}, we introduced a new approach to these problems that allowed us to show that, except for the 9 exceptional  unweighted models  
of \cite{BBMR15,BBMR17}, the generating series of genus one  unweighted models 
with infinite groups are $x$- and $y$-differentially transcendental. This reproves and generalizes the results of \cite{KurkRasch}.  Furthermore our results allowed us to show that the 9 exceptional  series are not holonomic but are $x$- and $y$-differentially algebraic, recovering some of the results of \cite{BBMR15,BBMR17}.  
It is worth mentioning that there are several results in the literature about the behavior of  $Q(x,y,t)$ with respect to the variable $t$. For instance, in \cite{MR09}, Mishna and Rechnitzer showed that $Q(1,1,t)$ is not $t$-holonomic for two of  the 5
 genus zero unweighted models
and in \cite{MelcMish}, Melczer and Mishna showed that this remained true for all 5 of the genus zero  unweighted models.  
On the other hand, Bostan, Raschel and Salvy proved in \cite{BRS} that  $Q(0,0,t)$ is not $t$-holonomic for every genus one  unweighted  model   with an infinite group.
We also note that, in \cite{BBMR15, BBMR17}, it is shown that the generating series of the 9 exceptional  genus zero unweighted models
mentioned above are differentially algebraic in the variable $t$ as well.
 Finally, the first two authors proved in \cite{dreyfus2019length} that, if the generating series is $x$- or $y$-differentially transcendental, then it is $t$-differentially transcendental. Thus, although the present paper focuses on the $x$- and $y$-differential properties of $Q(x,y,t)$, it also gives information concerning  its $t$-differential properties.

In the present paper, we start from    the  5 unweighted  models corresponding to a genus zero kernel curve. These models arise from the following 5 sets of steps.
 \begin{equation}\label{the five step set}\tag{S}
 \begin{tikzpicture}[scale=.4, baseline=(current bounding box.center)]
\foreach \x in {-1,0,1} \foreach \y in {-1,0,1} \fill(\x,\y) circle[radius=2pt];
\draw[thick,->](0,0)--(-1,1);
\draw[thick,->](0,0)--(0,1);
%\draw[thick,->](0,0)--(1,1);
%\draw[thick,->](0,0)--(-1,0);
%\draw[thick,->](0,0)--(1,0);
%\draw[thick,->](0,0)--(-1,-1);
%\draw[thick,->](0,0)--(0,-1);
\draw[thick,->](0,0)--(1,-1);
\end{tikzpicture}\quad 
\begin{tikzpicture}[scale=.4, baseline=(current bounding box.center)]
\foreach \x in {-1,0,1} \foreach \y in {-1,0,1} \fill(\x,\y) circle[radius=2pt];
\draw[thick,->](0,0)--(-1,1);
%\draw[thick,->](0,0)--(0,1);
\draw[thick,->](0,0)--(1,1);
%\draw[thick,->](0,0)--(-1,0);
%\draw[thick,->](0,0)--(1,0);
%\draw[thick,->](0,0)--(-1,-1);
%\draw[thick,->](0,0)--(0,-1);
\draw[thick,->](0,0)--(1,-1);
\end{tikzpicture}\quad \begin{tikzpicture}[scale=.4, baseline=(current bounding box.center)]
\foreach \x in {-1,0,1} \foreach \y in {-1,0,1} \fill(\x,\y) circle[radius=2pt];
\draw[thick,->](0,0)--(-1,1);
\draw[thick,->](0,0)--(0,1);
\draw[thick,->](0,0)--(1,1);
%\draw[thick,->](0,0)--(-1,0);
%\draw[thick,->](0,0)--(1,0);
%\draw[thick,->](0,0)--(-1,-1);
%\draw[thick,->](0,0)--(0,-1);
\draw[thick,->](0,0)--(1,-1);
\end{tikzpicture}\quad
\begin{tikzpicture}[scale=.4, baseline=(current bounding box.center)]
\foreach \x in {-1,0,1} \foreach \y in {-1,0,1} \fill(\x,\y) circle[radius=2pt];
\draw[thick,->](0,0)--(-1,1);
\draw[thick,->](0,0)--(0,1);
\draw[thick,->](0,0)--(1,1);
%\draw[thick,->](0,0)--(-1,0);
\draw[thick,->](0,0)--(1,0);
%\draw[thick,->](0,0)--(-1,-1);
%\draw[thick,->](0,0)--(0,-1);
\draw[thick,->](0,0)--(1,-1);
\end{tikzpicture}\quad 
\begin{tikzpicture}[scale=.4, baseline=(current bounding box.center)]
\foreach \x in {-1,0,1} \foreach \y in {-1,0,1} \fill(\x,\y) circle[radius=2pt];
\draw[thick,->](0,0)--(-1,1);
\draw[thick,->](0,0)--(0,1);
%\draw[thick,->](0,0)--(1,1);
%\draw[thick,->](0,0)--(-1,0);
\draw[thick,->](0,0)--(1,0);
%\draw[thick,->](0,0)--(-1,-1);
%\draw[thick,->](0,0)--(0,-1);
\draw[thick,->](0,0)--(1,-1);
\end{tikzpicture} 
\end{equation}

We say that a weighted model arises from \eqref{the five step set} when this model is obtained by  choosing a set of steps in \eqref{the five step set} and by assigning  nonzero weights to this set of steps. One can show that the kernel curve of a weighted model arising from \eqref{the five step set} is still a genus zero curve.
Our main result may be stated as follows:

\begin{thmintro}\label{theo:maintheotransc}
If $0<t<1$ is transcendental\footnote{This assumption is used repeatedly in our proofs and is crucial in our proof of Proposition~\ref{prop:quotient1}.} and if  the weighted model  arises from \eqref{the five step set}, then  $Q(x,y,t)$ is $x$- and $y$-differentially transcendental.
\end{thmintro}

Our study generalizes the result of Mishna and Rechnitzer 
on the non holonomy of the complete generating series of the unweighted models of walks $\{NW,N,SE\}$ and $\{NW,NE,SE\}$ (see \cite[Theorem 1.1]{MR09}) and also the one by Melczer and Mishna (for the five cases).  
Our strategy of proof is  inspired by \cite[Chapter 6]{FIM}.  We  associate to each of the generating series of these weighted  models  a function meromorphic on $\CX$. These associated functions satisfy first order difference equations  of the form $y(qs) - y(s) = b(s)$ for a suitable $q \in \CX$ and $b(s)\in\C(s)$.   The  associated  functions are differentially transcendental if and only if the generating series are differentially transcendental.  We then use criteria stating that if these associated functions were differentially algebraic then  the $b(s)$ must themselves satisfy $b(s) = h(qs)-h(s) $ for some rational functions $h(s)$ on $\CX$.  This latter condition puts severe limitations on the poles of the $b(s)$ and, by analyzing the $b(s)$ that arise, we show that these restrictions are not met.  Therefore the generating series are not differentially algebraic, see Theorem \ref{theo:maintheotransc}. Note that some  unweighted models of walks in dimension three happen to be, after projection, equivalent to  two dimensional weighted models of walks \cite{BoBMKaMe-16,DuHoWa-16}.   We apply our theorem in this setting as well. We note that finding the difference equation $y(qs) - y(s) = b(s)$ and the remaining calculations involve only algebraic computations as is true in \cite{DHRS}. The general approach followed in the present work is inspired by \cite{DHRS} but the details are quite different and justify an independent exposition.

The rest of the paper is organized as follows.  In Section~\ref{sec1}, we first present the generating series attached to a weighted model of walks and we give some of their basic properties. We then introduce the kernel curves (they are algebraic curves associated to any model of walk in the quarter plane) and we state some of their properties. 
One of their main properties is that, for the weighted models 
arising from    \eqref{the five step set}, the kernel curves have genus zero and, hence, can be parameterized by birational maps from  $\P1(\CX)$. Such parameterizations, suitable for our needs, are given at the end of Section~\ref{sec1}.   In Section~\ref{sec3}, using these parameterizations, we attach to any model  some meromorphic functions on $\CX$ that satisfy simple $q$-difference equations of the form ${y(qs)-y(s) = b(s)}$ for some $b(s) \in \CX(s)$. Moreover, we prove that these meromorphic functions are differentially algebraic if and only if the generating series of the  associated models 
are differentially algebraic.  In addition, we present necessary conditions  on the poles of $b$ when these equations have differentially algebraic solutions.  In Section~\ref{sec4}, we show that these necessary conditions do not hold for the   weighted models arising from    \eqref{the five step set}.

\vskip 5 pt

\noindent \textbf{Acknowledgments} The authors would like to thank Kilian Raschel for pointing out many references related to this work. In addition, we would like to thank the anonymous referees for many useful comments and suggestions concerning this article.

\section{Weighted walks in the quarter plane: generating series, functional equation and kernel curve}
\label{sec1}

In this section, we consider a weighted walk with small steps in the quarter plane $\Z_{\geq 0}^{2}$ and the corresponding trivariate generating series $Q(x,y,t)$ as in the introduction. We first recall a functional equation satisfied by $Q(x,y,t)$. 
We then recall the definition of the so-called kernel curve associated to the walk under consideration and give its main properties when the steps set is one of the five steps sets listed in \eqref{the five step set}.

\subsection{Kernel and functional equation}\label{sec1.2}
 
The {\it kernel} of a weighted model  is defined by 
$$
K(x,y,t):=xy (1-t S(x,y))
$$
where 
$$
\begin{array}{lll}
S(x,y) &=&\sum_{(i,j)\in \{0,\pm 1\}^{2}} d_{i,j}x^i y^j\\
&=& A_{-1}(x) \frac{1}{y} +A_{0}(x)+ A_{1}(x) y\\
&= & B_{-1}(y) \frac{1}{x} +B_{0}(y)+ B_{1}(y) x,
\end{array}
$$ 
and $A_{i}(x) \in x^{-1}\Q[x]$, $B_{i}(y) \in y^{-1}\Q[y]$. 

The following result generalizes \cite[Lemma 4]{BMM}.

\begin{lem}\label{lem:funceqn}
The generating series $Q(x,y,t)$ satisfies the following functional equation:
\begin{equation} \label{eq:funcequ}
K(x,y,t)Q(x,y,t)=xy -F^{1}(x,t) - F^{2}(y,t)+td_{-1,-1} Q(0,0,t)
\end{equation}
 where 
$$
 F^{1}(x,t):= -K(x,0,t)Q(x,0,t), \ \ F^{2}(y,t):= -K(0,y,t)Q(0,y,t).$$
\end{lem}

\begin{proof}

As in \cite[Lemma 4]{BMM}, we proceed as follows. First, let us prove that if we do not consider the quadrant constraint, the functional equation would be ${(1-tS(x,y))Q(x,y,t)=1}$. Indeed, in this situation, if we write $Q(x,y,t)=\displaystyle \sum_{\ell =0}^{\infty} Q_{\ell}(x,y)t^{\ell}$, then $Q_{0}(x,y)=1$ and $Q_{\ell+1}(x,y)=S(x,y)Q_{\ell}(x,y)$. This is exactly $(1-tS(x,y))Q(x,y,t)=1$.  However, this formula does not take into account  the quadrant constraint. We need to withdraw the  walks that leave the $x$-axis (resp. $y$-axis), i.e. $ty^{-1}A_{-1}(x)Q(x,0,t)$ (resp. $tx^{-1}B_{-1}(y)Q(0,y,t)$). Since we withdraw two times the walks going from $(0,0)$ in south west, we have to add the term  $tx^{-1}y^{-1}d_{-1,-1} Q(0,0,t)$.  So $$(1-t{S(x,y)})Q(x,y,t)=1-ty^{-1}A_{-1}(x)Q(x,0,t)-tx^{-1}B_{-1}(y)Q(0,y,t)+ tx^{-1}y^{-1}d_{-1,-1} Q(0,0,t).$$
It now suffices to multiply by $xy$ the above equality.  \end{proof}

\subsection{The algebraic curve defined by the kernel}\label{sec:algcurvedefkernelBIS}
We recall that the affine curve $E_{t}$ defined by the kernel ${K(x,y,t)}$ is given by 
 $$
 E_t = \{(x,y) \in \C \times \C \ \vert \ K(x,y,t) = 0\}. 
 $$ 
  In Section~\ref{sec3} we show that the problem of showing that $Q(x,y,t)$ is $x$- and $y$-differentially transcendental can be reduced to understanding the relations among the poles of a rational function on $E_{t}$. When dealing with a rational function $b(s)$  on $\C$, one often needs to consider its behavior "as $s$ goes to infinity". Although this can frequently be finessed, it is convenient to add a point at infinity, constructing the complex projective line as defined below, and consider the behavior at this point. When dealing with rational functions on curves in the affine plane, their behavior, such as the appearance of poles, often depends on missing points "at infinity" and we will see that this is the case in Section~\ref{sec3}. To do this we must include the missing points at infinity and so it is useful to compactify such a curve by adding these  points. This can be done in several ways (see Remark~\ref{rem:compact} below) but, as in   \cite{DHRS},}
it will be useful to consider a compactification $\Etproj$ of $E_{t}$ in $\P1(\C) \times \P1(\C)$, which is called the kernel curve.

We first recall that $\P1(\C)$ denotes the complex projective line, which is the quotient of $\C \times \C \setminus \{(0,0)\}$ by the equivalence relation $\sim$ defined by 
$$
(x_{0},x_{1}) \sim (x_{0}',x_{1}') \Leftrightarrow \exists \lambda \in \C^{*},  (x_{0}',x_{1}') = \lambda (x_{0},x_{1}). 
$$
The equivalence class of $(x_{0},x_{1}) \in \C \times \C \setminus \{(0,0)\}$ is usually denoted by $[x_{0}:x_{1}] \in \P1(\C)$. The map 
$
x \mapsto  [x:1]
$ 
embeds $\C$ inside $\P1(\C)$. The latter map is not surjective: its image is $\P1(\C) \setminus \{[1:0]\}$; the missing point $[1:0]$  is usually denoted by $\infty$. 
  Now, we embed $E_{t}$  inside $\P1(\C) \times \P1(\C)$ via  ${(x,y) \mapsto ([x:1],[y:1])}$. The kernel curve $\Etproj$ is the closure of this embedding of $E_{t}$.  In other words, the kernel curve $\Etproj $ is the algebraic curve defined by 
$$
\Etproj = \{([x_{0}:x_{1}],[y_{0}:y_{1}]) \in \P1(\C) \times \P1(\C) \ \vert \ \overline{K}(x_0,x_1,y_0,y_1,t) = 0\}
$$
where $\overline{K}(x_0,x_1,y_0,y_1,t)$ is the following bihomogeneous polynomial
\begin{equation}\label{eq:kernelwalk}
\overline{K}(x_0,x_1,y_0,y_1,t)={x_1^2y_1^2K(\frac{x_0}{x_1},\frac{y_0}{y_1},t)}= x_0x_1y_0y_1 -t \sum_{i,j=0}^2 d_{i-1,j-1} x_0^{i} x_1^{2-i}y_0^j y_1^{2-j}. 
 \end{equation}

Since $\overline{K}(x_0,x_1,y_0,y_1,t)$ is quadratic in each of the variables, the curve $\Etproj$ is naturally endowed with two involutions $\iota_1,\iota_2$, namely the vertical and horizontal  switches of $\Etproj$ defined, for any $P=(x,y) \in \Etproj$, by  
$$
\{P,\iota_1(P)\} = \Etproj \cap (\{x\} \times \P1(\C))
\text{ and }
\{P,\iota_2(P)\} = \Etproj \cap (\P1(\C) \times \{y\})
$$
 (see Figure \ref{figiota}).
 Let us also define $$\sigma:=\iota_{2}\circ\iota_{1}.$$

\begin{rem}\label{rem:compact}
 There are several choices for the compactification of $E_{t}$. For instance, we could have compactified the curve $E_{t}$ in the complex projective plane $\mathbb{P}^2(\C)$ instead of $\P1(\C) \times \P1(\C)$ but, in this case, the compactification is not defined by a biquadratic polynomial so that the construction of the above-mentioned involutions in that situation is not so natural. 
\end{rem}

\begin{assumption}\label{assumption:oneofthefive}
From now on, we consider a weighted  model arising from \eqref{the five step set} and we fix a transcendental real number $0<t< 1$.\footnote{In this paper, we have assumed that the $d_{i,j}$ belong to $\Q$, but everything stays true if we assume that $d_{i,j}$ are positive real numbers and that $t$ is transcendental over the field $\Q(d_{i,j})$.}
\end{assumption}

\begin{prop}\label{lem:genuszeroKernelBIS}
The curve $\Etproj$ is an irreducible genus zero curve. 
\end{prop}

\begin{proof}
This is the analog of \cite[Lemmas~2.3.2, 2.3.10]{FIM},  where the case $t=1$ is considered.
\end{proof}

\begin{figure}
\begin{center}
\includegraphics{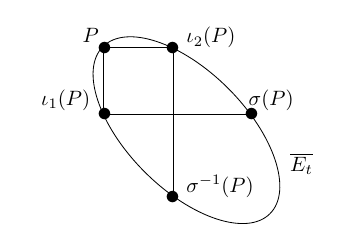}
\caption{The maps $\iota_{1},\iota_{2}$ restricted to the kernel curve $\Etproj $}\label{figiota}
\end{center}
\end{figure}

\subsection{Parametrization of $\Etproj$} 

Since $\Etproj$ has genus zero,  there is a rational parameterization of $\Etproj$, see \cite[Page 198, Ex.1]{Fultonalgcurves}, {\it i.e.}, there exists a birational map 
$$
\begin{array}{llll}
\phi :&\P1 (\C) &\rightarrow &\Etproj\\
& s &\mapsto & (x(s),y(s)).
\end{array}
$$
Proposition \ref{prop:parameterizationcompautoBIS} below gives such an explicit parametrization, which induces a bijection between $\P1 (\C) \setminus \phi^{-1}(\Omega)$ and $\Etproj \setminus \{\Omega\}$, where $\Omega=([0:1],[0:1])\in  \Etproj$. It is the analogue of \cite[Section 6.4.3]{FIM}, where the case $t=1$ is considered. The proof is similar for $t$ transcendental and the details are left to the reader. \par 

We first introduce some notations. For any $[x_0:x_1]$ and $[y_0:y_1]$ in $\P1(\C)$, we denote by $\Delta^x_{[x_0:x_1]}$ and $\Delta^y_{[y_0:y_1]}$ the discriminants of the degree two homogeneous polynomials given by $y \mapsto \overline{K}(x_0,x_1,y,t)$ and  $x \mapsto \overline{K}(x,y_0,y_1,t)$ respectively. We have 
\begin{multline}
\Delta^x_{[x_0:x_1]}=t^2 \Big( ( -  \frac{1}{t} x_0x_1 +d_{0,0}x_0x_1 + d_{1,0}x_0^2)^2  
 - 4d_{1,-1}x_0^2(d_{-1,1} x_1^2 + d_{0,1} x_0x_1 + d_{1,1}x_0^2) \Big)
 \nonumber
\end{multline}
and 
\begin{multline}
\Delta^y_{[y_0:y_1]}=t^2 \Big(-  \frac{1}{t} y_0y_1 +d_{0,0}y_0y_1+ d_{0,1}y_0^2)^2   
- 4d_{-1,1}y_0^2(d_{1,-1} y_1^2 + d_{1,0} y_0y_1 + d_{1,1}y_0^2) \Big).\nonumber
\end{multline}
Let us write 
$$
\Delta^x_{[x:1]}=\displaystyle \sum_{\ell=2}^{4}\alpha_{\ell}x^{\ell}
$$ 
and let $a_{1}=a_{2}=0,a_{3},a_{4}$ be the four roots of this polynomial. 
Similarly, let us write 
$$
\Delta^y_{[y:1]}=\displaystyle \sum_{\ell=2}^{4}\beta_{\ell}x^{\ell}
$$ 
and let $b_{1}=b_{2}=0,b_{3},b_{4}$ be the four roots of this polynomial. We have 
$$\begin{array}{lllllll}
\a_{2}(t)&=&1-2td_{0,0}+t^{2}d_{0,0}^{2}-4t^{2}d_{-1,1}d_{1,-1}& \ \ \  &\b_{2}(t)&=&1-2td_{0,0}+t^{2}d_{0,0}^{2}-4t^{2}d_{1,-1}d_{-1,1} \\
\a_{3}(t)&=&2t^{2}d_{1,0}d_{0,0}-2td_{1,0}-4t^{2}d_{0,1}d_{1,-1}& \ \ \  &\b_{3}(t)&=&2t^{2}d_{0,1}d_{0,0}-2td_{0,1}-4t^{2}d_{1,0}d_{-1,1}\\
\a_{4}(t)&=&t^{2}(d_{1,0}^{2}-4d_{1,1}d_{1,-1})& \ \ \  &\b_{4}(t)&=&t^{2}(d_{0,1}^{2}-4d_{1,1}d_{-1,1}).
\end{array}$$ 
Moreover, $a_{3}$, $a_{4}$, $b_{3}$ and $b_{4}$ are given by the following formulas 
$$
\begin{array}{|l|l|l|}\hline
&a_{3}&a_{4}\\[0.05in]\hline
\a_{4}(t)\neq 0&\left[\frac{-\a_{3}(t)-\sqrt{\a_{3}(t)^{2}-4\a_{2}(t)\a_{4}(t)}}{2\a_{4}(t)}:1\right]&\left[\frac{-\a_{3}(t)+\sqrt{\a_{3}(t)^{2}-4\a_{2}(t)\a_{4}(t)}}{2\a_{4}(t)}:1\right]
\\[0.05in]\hline
\a_{4}(t)= 0&[1:0] &[-\a_{2}(t):\a_{3}(t)]\\ \hline
&b_{3}&b_{4}\\[0.05in]\hline
\b_{4}(t)\neq 0&\left[\frac{-\b_{3}(t)-\sqrt{\b_{3}(t)^{2}-4\b_{2}(t)\b_{4}(t)}}{2\b_{4}(t)}:1\right]&\left[\frac{-\b_{3}(t)+\sqrt{\b_{3}(t)^{2}-4\b_{2}(t)\b_{4}(t)}}{2\b_{4}(t)}:1\right]
\\[0.05in]\hline
\b_{4}(t)= 0&[1:0] &[-\b_{2}(t):\b_{3}(t)]\\[0.05in] \hline
\end{array} $$

\begin{prop}\label{prop:parameterizationcompautoBIS}
An explicit parameterization $\phi= (x,y) : \P1 (\C) \rightarrow \Etproj$ is given by 
$$\phi(s) 
=\left(\dfrac{4\a_{2}(t)}{\sqrt{\a_{3}(t)^{2}-4\a_{2}(t)\a_{4}(t)}( s +\frac{1}{s}) -2\a_{3}(t)}, 
\dfrac{4\b_{2}(t)}{\sqrt{\b_{3}(t)^{2}-4\b_{2}(t)\b_{4}(t)}( \frac{s}{\lambda}+\frac{\lambda}{s}) -2\b_{3}(t)}\right)$$
for a certain $\lambda \in \C^{*}$.
Moreover we have (see Figure \ref{fig})
$$\begin{array}{lll}
x(0)=x(\infty)=a_{1},&x(1)=a_3,&x(-1)=a_4,\\
 y(0)=y(\infty)=b_1, &y(\lambda)=b_3, &y(-\lambda)=b_4,
\end{array} $$
where $a_1 = a_2 = [0:1]$  (resp. $b_1 = b_2 = [0:1]$). \end{prop}
\begin{figure}
\includegraphics[scale=0.4]{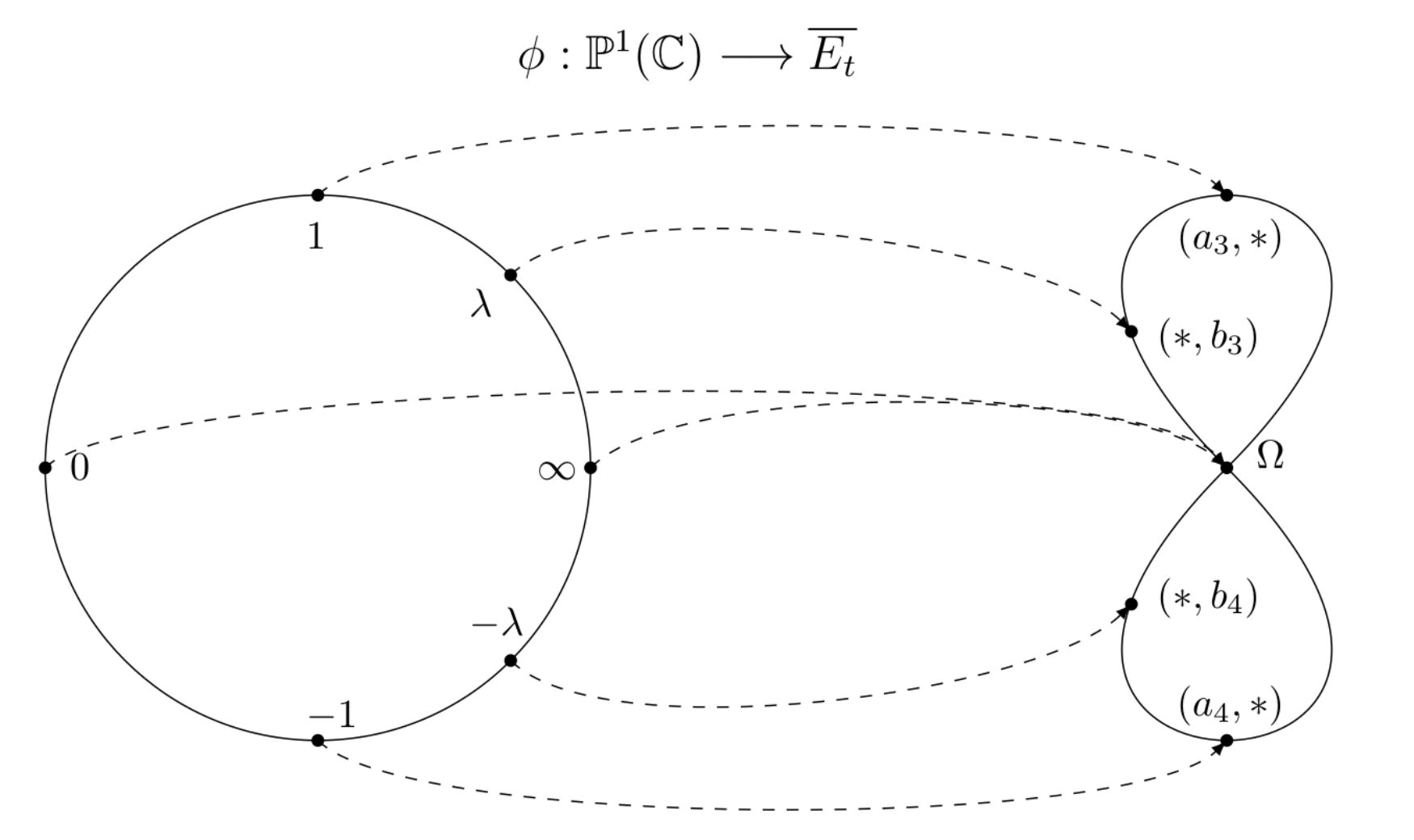}
\caption{The uniformization map}\label{fig}
\end{figure}

\begin{rem}
When $t=1$, we recover the uniformization of \cite[Section 6.4.3]{FIM}. Note that if we consider $x_3,x_4$ (resp. $y_3,y_4$) defined in \cite[Chapter 6]{FIM}, we have the equality of sets $\{a_3,a_4\}=\{x_3,x_4\}$ and $\{b_3,b_4\}=\{y_3,y_4\}$, but do not have necessarily $a_i=x_i$, $b_j=y_j$, with $3\leq i,j \leq 4$.
\end{rem}
 
The number
$$
q:=\lambda^{2}
$$ 
will be crucial in the rest of the paper. The following lemma determines $q$ up to its inverse.

\begin{prop}\label{lem:qBIS}
One of the two complex numbers  $\{q,q^{-1}\}$ is equal to 
\begin{equation}\label{eq1}
\dfrac{-1+d_{0,0}t-\sqrt{(1-d_{0,0}t)^{2}-4d_{1,-1}d_{-1,1}t^{2}}}{-1+d_{0,0}t+\sqrt{(1-d_{0,0}t)^{2}-4d_{1,-1}d_{-1,1}t^{2}}}.
\end{equation}
\end{prop}

\begin{proof}
Using the explicit formulas for $x(s)$ and $y(s)$, we get 
$$
\lim_{s\to 0}\frac{x(s)}{y(s)}=\dfrac{ \lambda\a_{2}(t)\sqrt{\b_{3}(t)^{2}-4\b_{2}(t)\b_{4}(t)}}{\b_{2}(t) \sqrt{\a_{3}(t)^{2}-4\a_{2}(t)\a_{4}(t)}}
\text{ and } 
\lim_{s\to 0}\frac{x(1/s)}{y(1/s)}=\dfrac{ \a_{2}(t)\sqrt{\b_{3}(t)^{2}-4\b_{2}(t)\b_{4}(t)}}{\lambda\b_{2}(t) \sqrt{\a_{3}(t)^{2}-4\a_{2}(t)\a_{4}(t)}}.
$$
But, $\frac{x(1/s)}{y(1/s)}=\frac{x(s)}{y(\iup_{1}(s))}$. So, the above two limits imply the following: 
$$
\lim_{s\to 0}\frac{y(\iup_{1}(s))}{y(s)}=q.
$$
Now, let us note that $y(s),y(\iup_{1}(s))$ equals to
$$
\dfrac{-x+d_{0,0}xt+d_{1,0}x^{2}t\pm\sqrt{(x-d_{0,0}xt-d_{1,0}x^{2}t)^{2}-4d_{1,-1}x^{2}t^{2}(d_{-1,1}+d_{0,1}x+d_{1,1}x^{2})}}{-2d_{-1,1}t-2d_{0,1}xt-2d_{1,1}x^{2}t},$$
with the shorthand notation $x=x(s)$. Since $x(s)$ tends to $0$ when $s$ goes to  $0$, we obtain the result.
\end{proof}

\begin{rem}
 One of the referees  remarked that for  the special case $d_{0,0}=0$, and $d_{1,-1}=d_{-1,1}=d$, the inverse of \eqref{eq1} becomes $$ 
\dfrac{-1+\sqrt{1-4d^{2}t^{2}}}{-1-\sqrt{1-4d^{2}t^{2}}}
=\dfrac{1-\sqrt{1-4d^{2}t^{2}}}{1+\sqrt{1-4d^{2}t^{2}}}=\dfrac{(1-\sqrt{1-4d^{2}t^{2}})^{2}}{4d^{2}t^{2}}= \frac{1-\sqrt{1-4d^{2}t^{2}}}{2d^{2}t^{2}}-1.$$
This expression is very similar to the generating series $\frac{1-\sqrt{1-4x}}{2x}$ of the Catalan numbers. Regrettably, we do not have, in general,  a combinatorial interpretation of $q$. 
\end{rem}

\begin{rem}
{ The uniformization is not unique.  More precisely, the possible uniformizations are of the form $\phi\circ h$, where $h$ is an homography. However, if one requires that $h$ fixes setwise $0,\infty$ then $q$ is uniquely defined up to its inverse.

 The real $q$ or $q^{-1}$ specializes for $t=1$ to the real $\rho^2$ in \cite[Page 178]{FIM}. In \cite[ (7.2.18) and Proposition 7.2.3]{FIM} it  is proved that the ratio of the argument of $\rho$  by $\pi$ is related to the angle between the tangent lines to the curve $\overline{E_1}$, the kernel curve at $t=1$, and the horizontal axis. This relation is obtained by a degeneracy argument from the genus $1$ case to the genus $0$ case.  More precisely, let  $\omega_3$ be  the period attached to the automorphism of the model of the walk in an  elliptic lattice $\Z\omega_1 + \Z \omega_2$ corresponding the elliptic kernel curve and  where $\omega_2$ is a  real period. Then, $\frac{\mathrm{arg}(\rho)}{\pi}$ is obtained by degeneracy of the fraction $\frac{\omega_3}{\omega_2}$ from the genus $1$ to the genus $0$ case. It is not completely obvious if these arguments pass to the situation where $t$ varies. In the zero drift situation, this has been done in \cite{fayolleRaschel}. In the general situation, it might be interesting to compute the rotation number $\frac{\omega_3(t)}{\omega_2(t)}$ of the real elliptic fibration (\cite[Page 82]{DuistQRT}) and to study its degeneracy. One could then expect that the ratio of the argument of $q$ by $2\pi$ is counting the number of rotations of the curve around the origin induced by the action of  the automorphism of the  model of the walk.}
\end{rem}
 
\begin{cor}\label{coro:qneqpm1}
We have $q\in \R\setminus \{\pm 1\}$.
\end{cor}

\begin{proof}
We first claim that $(1-d_{0,0}t)^{2}-4d_{1,-1}d_{-1,1}t^{2}>0$. We know that the $d_{i,j}$ are $\geq 0$, that the sum of the $d_{i,j}$ is equal to $1$ and that the  model is not included in $\{(0,0), (1,-1), (-1,1)\}$. Therefore, we have ${1>d_{0,0}+d_{1,-1}+d_{-1,1}}$, {\it i.e.},  $1-d_{0,0}>d_{1,-1}+d_{-1,1}$. Since $t\in ]0,1[$, we have $1-d_{0,0}t>1-d_{0,0}$. Thus, $(1-d_{0,0}t)^{2}> (1-d_{0,0})^{2}>(d_{1,-1}+d_{-1,1})^{2}$ and, hence, 
$$\begin{array}{lll}
(1-d_{0,0}t)^{2}-4d_{1,-1}d_{-1,1}t^{2}&>&(d_{1,-1}+d_{-1,1})^{2}-4d_{1,-1}d_{-1,1}t^{2} \\
& \geq &  (d_{1,-1}+d_{-1,1})^{2}-4d_{1,-1}d_{-1,1} = (d_{1,-1}-d_{-1,1})^{2}\geq 0.
\end{array} $$
This proves our claim. 

Now Proposition \ref{lem:qBIS} implies that $q$ is a real number $\neq 1$. Moreover, it also shows that $q= - 1$ if and only if $-1+d_{0,0}t=0$. But this is excluded because $1>d_{0,0}t$. 
\end{proof}

In particular, this implies that the birational maps $\sigma$ and $\tilde{\sigma}$ have infinite order (see also \cite{BMM,fayolleRaschel}). It follows that the group associated with these  models of walks, namely the group $\langle i_{1},i_{2} \rangle$ generated by $i_{1}$ and $i_{2}$, has infinite order (because  $\sigma$ is induced on $\Etproj$ by $i_{1} \circ i_{2}$, so if $\sigma$ has infinite order then $\langle i_{1},i_{2} \rangle$ has infinite order as well). Note that in \cite{BMM}, this was proved using a valuation argument. Using the valuation of the successive elements $(i_{1} \circ i_{2})^{\ell}(f)$ for $\ell\in \Z$ and $f\in \Q(x,y)$, it was proved that $i_{1} \circ i_{2}$ could not be of finite order. Initially, the group of the weighted  model was defined as a group of  birational transformations of $\C^2$, generated by two involutions. This is the group studied in \cite{BMM}. It is a finite group if and only if the automorphism of the weighted  model  $\sigma$ is of finite order.

\section{Analytic continuation  and differential transcendence criteria}\label{sec3} 

The aim of this section is to give differential transcendence criteria adapted to the study of the generating series of the  weighted models arising from \eqref{the five step set}. Let us describe our strategy. 
 In Lemma~\ref{lem:funceqn}, we defined the auxiliary series    $F^{1}(x,t):= -K(x,0,t)Q(x,0,t), \ \ F^{2}(y,t):= -K(0,y,t)Q(0,y,t)$. Since it is obvious that $Q(x,y,t)$ converges  for $|x| < 1, |y|< 1, |t|<1$, we have the same conclusion for these former series as well. Using the parameterization $\phi=(x,y):\P1 (\CX) \rightarrow \Etproj$ given in the previous section, we can pull back these functions to functions
$$\tilde{F}^{1}(s)=F^1(x(s),t) \text{ and } \tilde{F}^{2}(s)=F^2(y(s),t)$$ 
analytic in a neighborhood of $0$ in $\P1(\C)$. Using the functional equation (\ref{eq:funcequ}), we will prove that $\tilde{F}^{1}(s)$ and $\tilde{F}^{2}(s)$ each satisfy very simple $q$-difference equations
$$    \tilde{F}^{1}(qs) - \tilde{F}^{1}(s) = \tilde{b}_1 (s), \hspace{.2in} \tilde{F}^{2}(qs) - \tilde{F}^{2}(s) = \tilde{b}_2 (s) ,$$
for suitable $\tilde{b}_1, \tilde{b_2} \in \C(s)$. This implies in particular that $\tilde{F}^{1}(s)$ and $\tilde{F}^{2}(s)$ can be continued into meromorphic functions on all of $\C$. A result of Ishizaki, see \cite{Ishi}, implies that if either $\tilde{F}^{1}(s)$ or $\tilde{F}^{2}(s)$ are $s$-differentially algebraic then they must be in $\C(s)$ and results from the theory of linear $q$-difference equations allow us to detect this {\it via} the partial fraction decomposition of $\tilde{b}_1$ and $\tilde{b}_2$.  
In addition, we will show that $\tilde{F}^{1}(s)$ (resp.  $\tilde{F}^{2}(s)$) is $s$-differentially algebraic  if and only if   $Q(x,0,t)$ (resp. $Q(0,y,t)$) is $x$-differentially algebraic (resp. $y$-differentially algebraic).
We will therefore be able to reduce the question of whether $Q(x,0,t)$ (resp. $Q(0,y,t)$) is $x$-differentially algebraic (resp. $y$-differentially algebraic) to seeing if the above mentioned conditions on the partial fraction decomposition of $\tilde{b}_1$ and $\tilde{b}_2$ hold.  This will be done in Section~\ref{sec4} where we will see that the latter conditions never hold.  We now turn to supplying the details of this brief sketch.

In this section, we continue to assume that Assumption \ref{assumption:oneofthefive} holds true. 

\subsection{Functional equation}

We let 
$\phi = (x,y): \P1 (\C) \rightarrow \Etproj$ be the parameterization of $\Etproj$ given in Proposition~\ref{prop:parameterizationcompautoBIS}. Straightforward calculations show that
\begin{itemize}
\item $\phi(0)=\phi(\infty)=([0:1],[0:1])$;
\item $x(\iup_{1}(s))=x(s)$ where $\iup_1(s)=\frac{1}{s}$; 
\item $y(\iup_{2}(s))=y(s)$ where $\iup_2(s) = \frac{q}{s}=\frac{\lambda^{2}}{s}$;
\item $\tilde{\sigma}(s)=q s$ where $\tilde{\sigma} =\iup_2 \circ \iup_1$.
\end{itemize}

In particular, we have that $\tilde{\iota}_k\circ \phi = \phi\circ\iota_k$ and $\tilde{\sigma}\circ\phi = \phi\circ\sigma$ which will allow the following computations.

Recall the functional equation \eqref{eq:funcequ}: $$
 K(x,y,t)Q(x,y,t)=xy-F^{1}(x,t) - F^{2}(y,t)+td_{-1,-1} Q(0,0,t).
$$
This equation is a formal identity but for  $|x| < 1$ and $|y| <1$, the series $Q(x,y,t)$, $F^{1}(x,t)$ and $F^{2}(y,t)$ are convergent.  Using our parameterization of $\Etproj$, we will show how we can pull back these convergent series and analytically continue them to meromorphic functions on  $\CX$ satisfying simple $q$-difference equations. 

The set $V = \{([x:1],[y:1]) \in \Etproj \ \vert \ |x|, |y| <1 \}$ is an open neighborhood of $([0:1],[0:1])$ in $\Etproj$ for the analytic topology, and, for all $(x,y) \in V$, we have 
\begin{equation}\label{eq:funcequaonthecurve2}
0=xy -F^{1}(x,t)  -F^{2}(y,t)+td_{-1,-1} Q(0,0,t).
\end{equation}

Since $\phi(0)=\phi(\infty)=([0:1],[0:1])$, there exists $U \subset \P1(\C)$ which is the union of two small open discs centered at $0$ and $\infty$ such that $\phi(U) \subset V$. 

For any $s \in U$, we set $\breve{F}^{1}(s)=F^{1}(x(s),t)$ and $\breve{F}^{2}(s)=F^{2}(y(s),t)$. Then, $\breve{F}^{1}$ and $\breve{F}^{2}$ are meromorphic functions over $U$ and 
\eqref{eq:funcequaonthecurve2} yields, for all $s \in U$,  
\begin{equation}\label{eq:funcequaonthecurve3}
0=x(s)y(s)-\breve{F}^{1}(s)  -\breve{F}^{2}(s)+td_{-1,-1} Q(0,0,t).
\end{equation}

Replacing $s$ by $\iup_2(s)$ in \eqref{eq:funcequaonthecurve3}, we obtain, for all $s$ close to $0$ or $\infty$, (in what follows, we use $x(\iup_1(s))=x(s)$, $y(\iup_2(s))=y(s)$, $\breve{F}^{1}(\iup_1(s))=\breve{F}^{1}(s)$ and $\breve{F}^{2}(\iup_2(s))=\breve{F}^{2}(s)$) 
\begin{align}
0 &= x(\iup_2(s))y(\iup_2(s)) - \breve{F}^{1}(\iup_2(s)) - \breve{F}^{2}(\iup_2(s))+td_{-1,-1} Q(0,0,t) \nonumber \\
 &= x(\iup_{1}(\iup_2(s)))y(s)  - \breve{F}^{1}(\iup_{1}(\iup_2(s))) - \breve{F}^{2}(s)+td_{-1,-1} Q(0,0,t) \nonumber \\
& =   x(q^{-1}s)y(s)  - \breve{F}^{1}(q^{-1}s) - \breve{F}^{2}(s)+td_{-1,-1} Q(0,0,t). \label{eqn:equgroupaction} 
\end{align}
 Subtracting \eqref{eq:funcequaonthecurve3} from \eqref{eqn:equgroupaction}, and then replacing $s$ by $qs$, we obtain, for all $s$ close to $0$ or $\infty$, 
\begin{align}
 \breve{F}^1(qs)- \breve{F}^1(s) & =  {(x(qs)-x(s))y(qs).} \label{eq:firstfuncequation} 
\end{align}
\begin{rem}
If we set $t=1$ and replace $\breve{F}^{1}$ by $\frac{-\breve{F}^{1}}{K(0,y,t)}$, then a similar argument leads to another functional equation which is the one given in \cite[Theorem 6.4.1]{FIM}.
\end{rem}
Similarly, replacing $s$ by $\iup_1(s)$ in \eqref{eq:funcequaonthecurve3}, we obtain, for all $s$ close to $0$ or $\infty$, 
\begin{align}
0 &= x(\iup_1(s))y(\iup_1(s))  - \breve{F}^{1}(\iup_1(s)) - \breve{F}^{2}(\iup_1(s))+td_{-1,-1} Q(0,0,t) \nonumber \\
& =  x(s)y(\iup_{2}(\iup_1(s)))  - \breve{F}^{1}(s) - \breve{F}^{2}(\iup_{2}(\iup_{1}(s)))+td_{-1,-1} Q(0,0,t) \nonumber \\
& =  x(s)y(qs)  - \breve{F}^{1}(s) - \breve{F}^{2}(qs)+td_{-1,-1} Q(0,0,t). \label{eqn:equgroupaction2} 
\end{align}
 Subtracting  \eqref{eqn:equgroupaction2} from \eqref{eq:funcequaonthecurve3}, we obtain, for all $s$ close to $0$ or $\infty$, 
\begin{align}
 \breve{F}^2(qs)- \breve{F}^2(s)  = {x(s)(y(qs)-y(s))} . \label{eq:secondfuncequation} 
\end{align}

We let $\widetilde{F}^{1}$ and $\widetilde{F}^{2}$ be the restrictions of $\breve{F}^{1}$ and $\breve{F}^{2}$ to a small disc around $0$. They satisfy the functional equations \eqref{eq:firstfuncequation} and \eqref{eq:secondfuncequation} for $s$ close to $0$. 
Since $|q|\notin\{ 0,1\}$, this implies that each  of the functions $\widetilde{F}^{1}$ and $\widetilde{F}^{2}$ can be continued to a meromorphic function on $\C$ with \eqref{eq:firstfuncequation} satisfied for all $s\in \C$. 
Note that there is a priori no reason why, in the neighborhood of $\infty$, these functions should coincide with the original functions $\breve{F}^1$ and $\breve{F}^2$. 

\subsection{Application to differential transcendence}

In this subsection, we derive differential transcendency criteria for $x\mapsto Q(x,0,t)$ and $y\mapsto Q(0,y,t)$.  They are based on the fact that the related functions $\widetilde{F}^{1}$ and $\widetilde{F}^{2}$ satisfy  difference equations.

\begin{defi} Let $(E,\delta)\subset (F,\delta)$ be differential fields, that is, fields equipped with a map $\delta$ that satisfies $\delta (a+b)=\delta (a)+\delta (b)$ \text{and} $\delta (ab)=a\delta(b)+\delta(a)b$. 
We say that $f\in F$ is {\rm  differentially algebraic} over $E$ if it satisfies a non trivial algebraic differential equation with coefficients in $E$, {\it i.e.}, if for some $m$ there exists a nonzero polynomial ${P(y_0, \ldots , y_m) \in E[y_0, \ldots , y_m]}$ such that
$$
P(f,\delta(f), \ldots, \delta^m(f)) = 0.
$$
We say that $f$ is {\rm holonomic} over $E$ if in addition, the polynomial is linear. We say that  $f$  is {\rm differentially transcendental} over $E$ if it is not differentially algebraic.  
\end{defi}

\begin{prop}\label{prop:hyper}
The series $x\mapsto Q(x,0,t)$ is differentially algebraic over  $(\C(x), \frac{d}{dx})$ if and only if $\widetilde{F}^1$ is differentially algebraic over $(\C(s), \frac{d}{ds})$. The series $y\mapsto Q(0,y,t)$ is differentially algebraic over  $(\C(y), \frac{d}{dy})$ if and only if $\widetilde{F}^2$ is differentially algebraic over $(\C(s), \frac{d}{ds})$.  
\end{prop}

\begin{proof}  
This follows from Lemmas 6.3 and 6.4 of \cite{DHRS}, since we go from $x\mapsto Q(x,0,t)$ to $\widetilde{F}^1$ by a variable change which is algebraic (and therefore differentially algebraic). The proof for $\widetilde{F}^2$ is similar.  \end{proof}

Consequently, we only need to study $\widetilde{F}^1$ and $\widetilde{F}^2$. 
Recall that  they belong to the field  $\calM er(\C)$ of meromorphic functions on $\C$. Using a result due to Ishizaki \cite[Theorem 1.2]{Ishi} (see also \cite[Proposition 3.5]{HS}, where a Galoisian proof of Ishizaki's result is given), we get, for any $i\in \{1,2\}$, the following dichotomy\footnote{  Ishizaki's proof of his result proceeds by comparing behavior at various poles and uses growth results from Wiman-Valiron Theory. The approach  of \cite{HS} avoids the growth considerations and is more algebraic. A slightly weaker result, in the spirit of the considerations of \cite{DHRS}, would suffice to establish this dichotomy, see \cite[Corollary 3.2, Proposition 6.4]{HS} or \cite{Hard}.}:
\begin{itemize}
\item either $\widetilde{F}^i \in \C(s)$, or
\item  $\widetilde{F}^i $ is differentially transcendental over $\C(s)$.
\end{itemize}
\begin{rem}
1. Note that the fact that $\widetilde{F}^i$ is meromorphic on $\C$ is essential. For instance, if $q>1$, the Theta function $\theta_{q}(s)=\sum_{n\in\Z}{q}^{-n(n-1)/2}s^n$ is meromorphic on $\C^{*}$, is not rational and is differentially algebraic as it is shown for instance in \cite[Corollary 3.4]{HS}.\\
2. Combining  Ishizaki's dichotomy with the result of Mishna and Rechnitzer \cite{MR09}, and the result of Melczer and Mishna \cite{MelcMish},
on the non holonomy of the 
complete generating series of the unweighted  genus zero walks, one finds that these complete generating series are differentially transcendental, thus proving directly Theorem \ref{theo:maintheotransc} in the five unweighted cases. 
\end{rem}
So, we need to understand when $\widetilde{F}^i \in \C(s)$. We set 
$$
\widetilde{b}_{1}(s)=y(qs)( x(qs)-x(s)) \text{ and } \widetilde{b}_{2}(s)=x(s)(y(qs)-y(s)),
$$
so that the functional equations \eqref{eq:firstfuncequation} and \eqref{eq:secondfuncequation} can be restated as  
\begin{equation}\label{eq:eqfuncwithb}
\widetilde{F}^1(qs)- \widetilde{F}^1(s) = \widetilde{b}_{1}(s) \text{ and } \widetilde{F}^2(qs)- \widetilde{F}^2(s)=\widetilde{b}_{2}(s)
\end{equation}
for $s\in \C$. 

 \begin{lem}\label{lem6}\label{lem:equiv1}
For any $i \in \{1,2\}$, the following facts are equivalent:
\begin{itemize}
\item $\widetilde{F}^i \in \C(s)$;
\item there exists $f_{i}\in \C(s)$ such that $\widetilde{b}_{i}(s)=f_{i}(qs)-f_{i}(s)$.
\end{itemize}
 \end{lem}
 
\begin{proof}
If $\widetilde{F}^i \in \C(s)$ then \eqref{eq:eqfuncwithb} shows that $\widetilde{b}_{i}(s)=f_{i}(qs)-f_{i}(s)$ with $f_{i}=\widetilde{F}^i \in \C(s)$. Conversely, assume that there exists $f_{i}\in \C(s)$ such that $\widetilde{b}_{i}(s)=f_{i}(qs)-f_{i}(s)$. Using \eqref{eq:eqfuncwithb} again, we find that $(\widetilde{F}^i-f_{i})(s)=(\widetilde{F}^i-f_{i})(qs)$. 
Since the function $\widetilde{F}^i-f_{i}$ is meromorphic over $\C$,  we may expand it as a Laurent series at $s=0$: $\widetilde{F}^i-f_{i}=\sum_{\ell \geq \ell_0} a_{\ell}s^{\ell}$.  We then have $\sum_{\ell \geq \ell_0} a_{\ell}s^{\ell}=\sum_{\ell \geq \ell_0} a_{\ell}q^{\ell}s^{\ell}$ and since $q$ is not a root of unity, $\widetilde{F}^i-f_{i}\in \C$.
This ensures that $\widetilde{F}^i \in \C(s)$. 
\end{proof}

\begin{rem}\label{rem:decouplingfunction}
  In  \cite{BBMR17}, the authors introduce the  notion of decoupling functions, that is of functions $F(x) \in \Q(x,t)$ and $G(y) \in \Q(y,t)$ such that $xy= F(x)+ G(y)$
for $x, y$ satisfying $K(x,y,t) = 0$. It is easily seen that if $F$ and $G$ are decoupling functions, one has $$ \iota_2(xy)-xy=\iota_2(F(x))-F(x) \mbox { and } \iota_1(xy) -xy =\iota_1 (G(y))-G(y),$$ 
when  $K(x,y,t)=0$.  In our genus zero situation, composing the former identities with the uniformization yields  $\widetilde{b}_{i}(s)=f_{i}(qs)-f_{i}(s)$ where $f_1(s)=F(x(s)) \in \C(s)$ and ${f_2(s)= G(y(s)) \in \C(s)}$. Then, Lemma \ref{lem:equiv1} is essentially the same kind of results as \cite[Lemma 2]{BBMR17} but in the easier framework of a genus zero kernel curve.
\end{rem}

The following lemma is a consequence of the functional equation satisfied by $\widetilde{F}^1,\widetilde{F}^2$. See \cite[Corollary 3.2.5]{FIM}, or \cite[Proposition 3.10]{DHRS}, for similar results in the genus one case.
\begin{lem}\label{lem:equiv2}
The following properties are equivalent: 
\begin{itemize}
\item  $\widetilde{F}^1 \in \C(s)$;
\item  $\widetilde{F}^2 \in \C(s)$.
\end{itemize}
\end{lem}

\begin{proof}
Assume that $\widetilde{F}^1 \in \C(s)$.   Lemma~\ref{lem6} states that  there exists $f_{1}\in \C(s)$ such that ${\widetilde{b}_{1}(s)=f_{1}(qs)-f_{1}(s)}$. Note that  $\widetilde{b}_{1}(s)+\widetilde{b}_{2}(s)=(xy)(qs)-(xy)(s)$, so that we  have ${\widetilde{b}_{2}(s)=f_{2}(qs)-f_{2}(s)}$, with $xy(s)-f_{1}(s)=f_{2}(s)\in \C(s)$.    Lemma~\ref{lem6} implies that  $\widetilde{F}^2 \in \C(s)$. The converse is proved in   a similar way.
\end{proof}

\begin{thm}\label{thm:equivconditions}
The following properties are equivalent:
\begin{enumerate}
\item The series $Q(x,0,t)$ is differentially algebraic over $\C(x)$;\label{equivconditions1}
\item The series $ Q(x,0,t)$ is algebraic over $\C(x)$;\label{equivconditions2}
\item The series $Q(0,y,t)$ is differentially algebraic over $\C(y)$;\label{equivconditions3}
\item The series $Q(0,y,t)$ is algebraic over $\C(y)$;\label{equivconditions4}
\item There exists $f_{1}\in \C(s)$ such that $\widetilde{b}_{1}(s)=f_{1}(qs)-f_{1}(s)$; \label{equivconditions5}
\item There exists $f_{2}\in \C(s)$ such that $\widetilde{b}_{2}(s)=f_{2}(qs)-f_{2}(s)$. \label{equivconditions6}
\end{enumerate}
\end{thm}

\begin{proof} 
Assume that \eqref{equivconditions1} holds true. Proposition~\ref{prop:hyper} implies that $\widetilde{F}^{1}$ is differentially algebraic over $\C(s)$. Ishizaki's Theorem ensures that $\widetilde{F}^{1} \in \C(s)$. But $x : \P1(\C) \rightarrow \P1(\C)$ is locally (for the analytic topology) invertible at all but finitely many points of $\P1(\C)$ and the corresponding local inverses are algebraic over $\C(x)$. It follows that  $F^{1}(\cdot,t)$ can be expressed as a rational expression, with coefficients in $\C$, of an algebraic function, and, hence, is algebraic over $\C(x)$. Hence \eqref{equivconditions2} is satisfied. The fact that \eqref{equivconditions2} implies \eqref{equivconditions1} is obvious. 
The fact that \eqref{equivconditions3} is equivalent to \eqref{equivconditions4} can be shown in a similar manner to the equivalence of \eqref{equivconditions1} and \eqref{equivconditions2}. 
The fact that \eqref{equivconditions1} to \eqref{equivconditions4} are equivalent now follows from Lemma \ref{lem:equiv2} combined with \cite[Theorem 1.2]{Ishi}. 
The remaining equivalences follow from Lemma \ref{lem:equiv1}.
\end{proof}

So, to decide whether  $Q(x,0,t)$,  $Q(0,y,t)$ are differentially transcendental, we are led to the following problem:\\ \par 
\emph{Given $b \in \C(s)$, decide whether there exists $f\in \C(s)$ such that $b(s)=f(qs)-f(s)$.} \\\par 
  When such an $f$ exists, we say that $b$ is $q$-summable in $\C(s)$.  This problem is known as a $q$-summation problem and has been solved  by Abramov \cite{Abramov}.  This procedure  was recast in \cite{ChenSinger} in terms of 
 the so-called $q$-residues of $b$, which we now define.
 
  We begin by defining the {\it $q$-orbit} of $\beta \in \C^*$ to be $\beta q^{\Z} = \{ \beta\cdot  q^i \ | \ i \in \Z\}$. Given a rational function $b(s) \in \C(s)$ we may rewrite its partial fraction decomposition uniquely as
 \begin{equation}\label{eq:parfrac}b(s) = c + sp_1 + \frac{p_2}{s^r} + \sum_{i=1}^m \sum_{j=1}^{n_i} \sum_{\ell = 0}^{r_{i,j} }\frac{\alpha_{i,j,\ell}}{(s-q^\ell \cdot \beta_i)^j},
 \end{equation}
 where $ c \in \C, p_1,p_2 \in \C[s], m, n_i \in \Z_{\geq 0}$ are nonzero, $r, r_{i,j} \in \Z_{\geq 0}, \alpha_{i,j,\ell}, \beta_i\in \C$ and the $\beta_i$'s are nonzero and in distinct $q$-orbits.
 
 \begin{defi} (cf. \cite[Definition 2.7]{ChenSinger})  Let $b \in \C(s)$ be of the form (\ref{eq:parfrac}).  The sum 
 \[\sum_{\ell = 0}^{r_{i,j}} q^{-\ell\cdot j} \alpha_{i,j,\ell}\]
 is called the {\rm $q$-residue of b at the $q$-orbit of $\beta_i$ of multiplicity $j$} (this is called the {\it $q$-discrete residue} in \cite{ChenSinger}) and is denoted by $\mathrm{qres}(b, \beta_i,j)$. In addition, we call the constant $c$ the {\rm $q$-residue of $b$ at infinity} and denote it by $\mathrm{qres}(b, \infty)$.\end{defi}
 
 \begin{exa} Let $q = 2$ and 
 \[b(s) = 1 + s + \frac{s+2}{s^2} + \frac{3}{(s-1)^2} -\frac{12}{(s-2)^2} + \frac{1}{s-5}.\]
 We have $\mathrm{qres}(b, \infty) = 1$, $\mathrm{qres}(b, 1,2) = 2^0\cdot 3 + 2^{-1\cdot 2} (-12) = 0$, and $\mathrm{qres}(b, 5,1)= 1$.  All other $q$-residues are $0$. 
 \end{exa}
 
 One has the following criterion for $q$-summability.

 \begin{prop}(c.f. \cite[Proposition 2.10]{ChenSinger}) Let $b = f/g \in \C(x)$ be such that $f,g \in \C[x]$ with $\mathrm{gcd}(f,g) = 1$. Then $b$ is $q$-summable in $\C(s)$ if and only if the $q$-residues $\mathrm{qres}(b,\infty) = 0$ and $\mathrm{qres}(b, \beta, j) = 0$ for any multiplicity $j$ and any  $\beta \neq 0$ with $g(\beta) = 0, g(q^{\ell}\beta) \neq 0$ for every $\ell <0$.
 \end{prop}

 Applying this criteria to the above example we see that $b$ is not $q$-summable because $\mathrm{qres}(b, \infty) \neq 0$ as well as $\mathrm{qres}(b, 5,1)\neq 0$.  In fact, whenever an element $b\in \C(x)$ has a pole of order $m\geq 1$ at a point $\beta$ and no other pole of order $\geq m$ in the $q$-orbit of $\beta$, then a $q$-residue of multiplicity $m$ will be nonzero. We therefore have the following corollary (also a consequence of  results in \cite{Abramov}) which plays a crucial role  in the next section.
\begin{cor}\label{lem:ChenSinger} If $\beta  \in \C^{*}$  is a pole of $b\in \C(x)$ of order  $m \geq 1$ and if $b$ has no other pole of order $\geq m$ in the $q$-orbit of $\beta$, then $b$ is not $q$-summable, {\it i.e.}, there is no $f(s) \in \C(s)$ such that $b(s)=f(qs)-f(s)$.  \end{cor}

Using the parameterization $\phi:\P1 (\C) \rightarrow \Etproj$, we can translate this to give a criterion for the differential transcendence of 
$x\mapsto Q(x,0,t)$ and  $y\mapsto Q(0,y,t)$  over $\C(x)$ and $\C(y)$ respectively. We set (see Section~\ref{sec1} for notations)
$$
\mathbf{b}_{1}=\iota_{1}(y)(\iota_{2}(x)-x) \text{ and } \mathbf{b}_{2}=x(\iota_{1}(y)-y),
$$
so that we have 
$$
\widetilde{b}_{1}=\mathbf{b}_{1}\circ \phi \text{ and } \widetilde{b}_{2}=\mathbf{b}_{2} \circ \phi.
$$

\begin{prop} \label{prop:criteria}
We suppose that Assumption~\ref{assumption:oneofthefive} holds true and recall that $|q|\neq 1$. Let $b \in \C(x,y)$ be  a rational function on   $\Etproj$. Assume that $P \in \Etproj \setminus \{\Omega\}$ is a pole of $b$ of order $m \geq 1$ such that none of the $\sigma^i(P)$ with $i \in \Z \backslash \{0\}$ is a pole of $b$ of order $\geq m$, then $$b = \sigma(g) - g$$ has no solution $g \in \C(x,y)$ which restricts to a rational function on $\Etproj$.

In particular, if $\mathbf{b}_{2} = x(\iota_1(y)-y )$ satisfies this condition, then $x\mapsto Q(x,0,t)$,  (resp. $y\mapsto Q(0,y,t)$) is differentially transcendental over $\C(x)$ (resp. differentially transcendental over $\C(y)$). 
\end{prop}

\begin{proof} We know that the parameterization $\phi =(x,y):\P1 (\C)\rightarrow \Etproj$ that we have constructed, induces an isomorphism between $\P1(\C) \setminus \{0,\infty\}$ and $\Etproj \setminus \{\Omega\}$.  
If $s_0 \in \P1 (\C)\setminus \{0,\infty\}$ is such that $\phi(s_{0})=P$, then $s_{0}$ is a pole of order $m \geq 1$ of $b \circ \phi$ such that none of the $\tilde{\sigma}^i(s_{0})  = q^i{s}$ with $i \in \Z \backslash \{0\}$ is a pole of $b \circ \phi$ of order $\geq m$. 
If $g \in \C(x,y)$ restricts to a rational function on $\Etproj$ and satisfies $b = \sigma(g) - g$, then $f = g \circ \phi$ would satisfy $b(s)=f(qs)-f(s)$ contradicting Lemma~\ref{lem:ChenSinger}.

If $\mathbf{b}_{2} = x(\iota_1(y)-y)$ satisfies the condition of the Proposition, then $\mathbf{b}_{2} = \sigma(g) -g$ has no solution $g$ that is a rational function on $\Etproj$. Pulling this back to $\P1 (\C)$, we see that for $\widetilde{b}_2(s) = \mathbf{b}_{2} \circ \phi (s) = x(s)(y(1/s)-y(s))$, the equation $\widetilde{b}_2(s) = f(qs) - f(s)$ has no solution in $\C(s)$. Theorem~\ref{thm:equivconditions} yields our conclusion.\end{proof}

 Finally we note that given a fixed family of probabilities $(d_{i,j})$, the algorithms \cite{Abramov,ChenSinger} permit us to decide if the generating series is differentially algebraic or not. In Section~\ref{sec4}, we will prove an unconditional statement, that is, for every set of probabilities $d_{i,j}$, the generating series is differentially transcendental. Note that this kind of result may a priori not be obtained via the above mentioned algorithms, since the generating series depends on parameters (the probabilities $d_{i,j}$) and it is not clear how to make the algorithms give information about arbitrary specializations of the parameters.
 
\section{Differential transcendence: main result}\label{sec4}
 
In this section, we will prove the main result of this paper:  
 
\begin{thm}\label{theo:maintheotransc}
We suppose that Assumption~\ref{assumption:oneofthefive} is satisfied. Then, the functions $x\mapsto Q(x,0,t)$ and $y\mapsto Q(0,y,t)$ are differentially transcendental over $\C(x)$ and $\C(y)$ respectively.
\end{thm}

\begin{rem} (i)  Models of walks in three dimensions in the octant have been recently studied. In \cite{BoBMKaMe-16,DuHoWa-16}, the authors study such unweighted  models having at most six steps. Among the non trivial $35548$ models, $527$ are equivalent to weighted models of walks in the quarter plane,  in the sense of \cite[Definition 2]{BoBMKaMe-16} and  Assumption~\ref{assumption:oneofthefive} is satisfied for $69$ such models, see \cite[Section 3]{DuHoWa-16}. For these  models of two dimensional walks our results apply. For example, in \cite{DuHoWa-16}, the authors prove that one of the three dimensional unweighted  models of a walk in the octant is equivalent to the following weighted  model of a two dimensional  walk of genus zero:
\vspace{0.5cm}
\begin{center}\begin{tikzpicture}[scale=.4, baseline=(current bounding box.center)]
\foreach \x in {-1,0,1} \foreach \y in {-1,0,1} \fill(\x,\y) circle[radius=2pt];
\draw[thick,->](0,0)--(-1,1);
\draw[thick,->](0,0)--(0,1);
\draw[thick,->](0,0)--(1,-1);
\put(-30,15){{$\small{1/2}$}}
\put(-5,15){{$\small{1/4}$}}
\put(15,-15){{$\small{1/4}$}}
\end{tikzpicture}\end{center}

\vspace{.1in}

(ii)  Combining Theorem \ref{theo:maintheotransc} with Remark \ref{rem:decouplingfunction}, we have proved that in the genus zero situation there are no decoupling functions.

\end{rem}

The proof of Theorem~\ref{theo:maintheotransc} will be given at the very end of this section. Our strategy will be to use Proposition~\ref{prop:criteria}. So, we begin by collecting information concerning the poles of $\mathbf{b}_{2} = x(\iota_{1}(y)-y)$.  
 
 \subsection{Preliminary results concerning the poles of $\mathbf{b}_{2}$}

We write 
$$
\mathbf{b}_{2}=x(\iota_{1}(y)-y)
$$ 
in the projective coordinates $([x_0,x_1],[y_0,y_1])$ with $x=\frac{x_{0}}{x_{1}}$ and $y=\frac{y_{0}}{y_{1}}$. We note that ${\Omega=([0:1],[0:1])}$ is not a pole of $\mathbf{b}_{2}$. Since we want to compute the poles of $\mathbf{b}_{2}$, it is natural to start with the poles of $xy$. Therefore let us focus our attention on the points ${([x_{0}:x_{1}], [y_0:y_1])}$ of $\Etproj$ corresponding to the equation $x_1y_1=0$, namely:
\begin{align*}
P_1 =([1:0], [\beta_0:\beta_1]), & P_2 =\iota_1(P_1)=([1:0], [\beta_0':\beta_1']), \\ 
Q_1=([\alpha_0:\alpha_1],[1:0]), & Q_2=\iota_2(Q_1)=([\alpha_0':\alpha_1'],[1:0]).
\end{align*}
Since $P_{1},P_{2}\in \Etproj$, to compute $[\beta_0:\beta_1]$ and $[\beta'_0:\beta'_1]$, we have to solve $\overline{K}(1,0,y_0,y_1,t)=0$. We then find that  $[\beta_0:\beta_1]$ and $[\beta'_0:\beta'_1]$ are the roots in $\P1(\C)$ of the homogeneous polynomial in $y_{0}$ and $y_{1}$ given by  
$$
d_{1,-1}y_1^2 +d_{1,0}y_0y_1 +d_{1,1}y_0^2=0.
$$ 
Similarly, the $x$-coordinates $[\alpha_0:\alpha_1]$ and $[\alpha'_0:\alpha'_1]$ of $Q_1$ and $Q_2$ are the roots in $\P1(\C)$ of the homogeneous polynomial in $x_{0}$ and $x_{1}$ given by  
$$
d_{-1,1}x_1^2 +d_{0,1}x_0x_1 +d_{1,1}x_0^2=0.
$$ 
 Although the  following Lemma already appears in \cite[Lemma 4.11]{DHRS}, we give its proof to be self-contained.

\begin{lemma}\label{lemma:divisorsecondmember}
The set of poles of ${\mathbf{b}_{1}}= \iota_1(y)\left(\sigma (x)-x\right)$ in $\Etproj$ is contained in 
$$\mathcal{S}_{1}=\{\iota_{1}(Q_1),\iota_{1}(Q_2), P_1,P_2,\sigma^{-1}(P_1),\sigma^{-1}(P_2)\}.$$
Similarly, the set of poles of ${\mathbf{b}_{2}}=x(\iota_1(y)-y)$ in $\Etproj$ is contained in 
$$\mathcal{S}_{2}=\{P_1,P_2, Q_1,Q_2,\iota_1(Q_1),\iota_1(Q_2)\}=\{P_1,P_2,Q_1,Q_2,\sigma^{-1}(Q_1),\sigma^{-1}(Q_2)\}.$$
Moreover, we have
\begin{equation}\label{equ:secondmembreform}
({\mathbf{b}_{2}})^2=\frac{ x_0^2 \Delta^x_{[x_0:x_1]}}{x_1^2(\sum_{i=0}^2  x_0^i x_1^{2-i}td_{i-1,1} )^2 }.
\end{equation}  
\end{lemma}

 \begin{proof}
The proofs of the assertions about the location of the poles of ${\mathbf{b}_{1}}$ and ${\mathbf{b}_{2}}$ are straightforward. 
Let us prove  \eqref{equ:secondmembreform}. By definition,  the $y$ coordinates of $ \iota_1 ( \frac{y_0}{y_1} )$ and $\frac{y_0}{y_1} $ are the two roots of the polynomial $y\mapsto \overline{K}(x_0,x_1,y,t)$. The square of their difference equals to the discriminant divided by the square of the leading term,  that is,
$$
\left(  \iota_1 ( \frac{y_0}{y_1} ) -  \frac{y_0}{y_1} \right)^2= \frac{ \Delta^x_{[x_0:x_1]}}{(\sum_i  x_0^i x_1^{2-i}td_{i-1,1} )^2 }.
$$
Therefore, we find
$$
{\mathbf{b}_{2}}\left(\frac{x_0}{x_1},\frac{y_0}{y_1}\right)^2=\frac{ x_0^2 \Delta^x_{[x_0:x_1]}}{x_1^2(\sum_i  x_0^i x_1^{2-i}td_{i-1,1} )^2 }.
$$
\end{proof}

To apply Proposition \ref{prop:criteria} we now need to separate the orbits. Let us begin  with  $P_{1}$ and $P_{2}$ (resp. $Q_{1}$ and $Q_{2}$). In what follows, we will use the equivalent relation $\sim$ on $\Etproj$ defined, for $P,Q\in \Etproj$, by 
$$P\sim Q \Leftrightarrow \exists \ell\in \Z, \sigma^{\ell}(P)=Q.
$$

\begin{prop}\label{prop:quotient1}
 If $P_{1}\neq P_{2}$, then one of the following properties holds:   
\begin{itemize}
\item $P_{1}\not\sim P_{2}$;
\item $d_{0,1}=d_{1,1}=0$. 
\end{itemize} 
If $Q_{1}\neq Q_{2}$, then one of the following properties holds: 
\begin{itemize}
\item $Q_{1}\not\sim Q_{2}$;
\item $d_{1,0}=d_{1,1}=0$. 
\end{itemize} 
\end{prop}

\begin{proof}
We only prove the statement for the $P_{i}$, the proof for the $Q_{j}$ being similar. Let $p_{1},p_{2} \in \C^{*}$ be such that $\phi(p_{1})=P_{1}$ and $\phi(p_{2})=P_{2}$.  Recall that Proposition \ref{prop:parameterizationcompautoBIS} ensures that 
$$\begin{array}{lllllll}
\a_{2}&=&1-2td_{0,0}+t^{2}d_{0,0}^{2}-4t^{2}d_{-1,1}d_{1,-1} \\
\a_{3}&=&2t^{2}d_{1,0}d_{0,0}-2td_{1,0}-4t^{2}d_{0,1}d_{1,-1}\\
\a_{4}&=&t^{2}(d_{1,0}^{2}-4d_{1,1}d_{1,-1})
\end{array} $$
and that, according to
 Proposition \ref{lem:qBIS}, one of the two complex numbers $q$ or $q^{-1}$ is equal to 
$$ 
\dfrac{-1+d_{0,0}t-\sqrt{(1-d_{0,0}t)^{2}-4d_{1,-1}d_{-1,1}t^{2}}}{-1+d_{0,0}t+\sqrt{(1-d_{0,0}t)^{2}-4d_{1,-1}d_{-1,1}t^{2}}}.
$$ The explicit formula for $\phi$ given in Proposition \ref{prop:parameterizationcompautoBIS} shows that $p_{1}$ and $p_{2}$ are the roots of 
$$-\sqrt{\alpha_{3}^{2}-4\alpha_{2}\alpha_{4}}X^{2}+2\alpha_{3}X-\sqrt{\alpha_{3}^{2}-4\alpha_{2}\alpha_{4}}=0.$$
So, we have (for suitable choices of the complex square roots\footnote{{Since $p_i$ is chosen so that $\phi(p_i)=P_i$, we need to take the square roots consistent with this selection. }}) 

$$
p_{1}= \frac{-\alpha_{3} - 2\sqrt{\alpha_{2}\alpha_{4}}}{-\sqrt{\alpha_{3}^{2}-4\alpha_{2}\alpha_{4}}}
\text{ and }
p_{2}= \frac{-\alpha_{3} + 2\sqrt{\alpha_{2}\alpha_{4}}}{-\sqrt{\alpha_{3}^{2}-4\alpha_{2}\alpha_{4}}}.
$$

Assume that $P_{1}\sim P_{2}$. Then, there exists $\ell\in \Z^{*}$ such that   
$
\frac{p_{1}}{p_{2}} =q^{\ell}
$
($\ell \neq 0$ because $P_{1} \neq P_{2}$). Using the above formulas for $p_{1}$, $p_{2}$ and $q$ and replacing $\ell$ by $-\ell$ if necessary, this can be rewritten as: 
\begin{equation}\label{s0s1}
\frac{-\alpha_{3}- 2\sqrt{\alpha_{2}\alpha_{4}}}{-\alpha_{3}+ 2\sqrt{\alpha_{2}\alpha_{4}}} =\left(\frac{-1+d_{0,0}t-\sqrt{(1-d_{0,0}t)^{2}-4d_{1,-1}d_{-1,1}t^{2}}}{-1+d_{0,0}t+\sqrt{(1-d_{0,0}t)^{2}-4d_{1,-1}d_{-1,1}t^{2}}}\right)^{\ell}. 
\end{equation}

Recall that $t$ is transcendental. We shall treat $t$ as a variable and both sides of (\ref{s0s1}) as functions of the variable $t$, algebraic over $\Q(t)$. Formula \eqref{s0s1} shows that these algebraic functions coincide at some transcendental number, therefore they are equal. 

We now consider these algebraic functions near $0$ (we choose an arbitrary branch) and will derive a contradiction by proving that they have different behaviors at $0$. \par 
If $d_{1,1} \neq 0$, then, considering the Taylor expansions at $0$ in \eqref{s0s1}, we obtain, up to replacing  $\ell$ by $-\ell$ if necessary: 
$$
\frac{d_{1,0}-\Delta_1}{d_{1,0}+\Delta_1} + O(t) = \left(\frac{1}{t^2}\left(\frac{1}{d_{1,-1}d_{-1,1}} + O(1/t)\right)\right)^{\ell}
$$
where $\Delta_1$ is some square root of $d_{1,0}^{2}-4d_{1,1}d_{1,-1}$, and $d_{1,0}-\Delta_1$ and $d_{1,0}+\Delta_1$ are not $0$ because $d_{1,1} \neq 0$ ({note that, by Assumption \ref{assumption:oneofthefive}, we have $d_{1,-1}d_{-1,1}\neq 0$}). This equality is impossible. 

If $d_{1,1}= 0$, then \eqref{s0s1} gives   
$$
t\frac{d_{0,1}d_{1,-1}}{d_{1,0}}+O(t^{2})
=
 \left(\frac{1}{t^2}\left(\frac{1}{d_{1,-1}d_{-1,1}} + O(1/t)\right)\right)^{\ell} 
$$ 
(note that we have $d_{1,0}\neq 0$ because $P_{1}\neq P_{2}$). This implies $d_{0,1}= 0$ and conclude{s} the proof. 
\end{proof}

\begin{prop}\label{prop:quotient2}
Assume that $d_{1,1}\neq 0$. Then, for any $i,j \in \{1,2\}$, we have $P_{i}\not\sim  Q_{j}$.
\end{prop}

\begin{proof}
Let $p_{i},q_{j}\in \C^{*}$ be such that $\phi(p_{i})=P_{i}$ and $\phi(q_{j})=Q_{j}$. 
As seen at the beginning of the proof of Proposition \ref{prop:quotient1}, we have (for suitable choices of the square roots) 
$$
p_{i}= \frac{-\alpha_{3} - 2\sqrt{\alpha_{2}\alpha_{4}}}{-\sqrt{\alpha_{3}^{2}-4\alpha_{2}\alpha_{4}}}.
$$
Similarly, we have (for suitable choices of the square roots) 
$$
q_{j}=\lambda \frac{-\beta_{3} - 2\sqrt{\beta_{2}\beta_{4}}}{-\sqrt{\beta_{3}^{2}-4\beta_{2}\beta_{4}}}.
$$

Suppose to the contrary that $P_{i}\sim  Q_{j}$. The condition $d_{1,1}\neq 0$ yields that $P_{i}\neq Q_{j}$. Then, there exists $\ell\in \Z^{*}$ such that   
$
\frac{p_{i}}{q_{j}}=q^{\ell}
$. 
Using the above formulas for $p_{i}$ and $q_{j}$, using Proposition~\ref{lem:qBIS} and replacing $\ell$ by $-\ell$ if necessary, this can be rewritten as:
\begin{equation}\label{s0s1bis}
\frac{\alpha_{3}+2\sqrt{\alpha_{2}\alpha_{4}}}{\sqrt{\alpha_{3}^{2}-4\alpha_{2}\alpha_{4}}}\frac{\sqrt{\beta_{3}^{2}-4\beta_{2}\beta_{4}}}{\beta_{3}+2\sqrt{\beta_{2}\beta_{4}}} =\left(\frac{-1+d_{0,0}t-\sqrt{(1-d_{0,0}t)^{2}-4d_{1,-1}d_{-1,1}t^{2}}}{-1+d_{0,0}t+\sqrt{(1-d_{0,0}t)^{2}-4d_{1,-1}d_{-1,1}t^{2}}}\right)^{\ell+\frac12}.
\end{equation}
As in the proof of Proposition \ref{prop:quotient1}, we can treat $t$ as a variable and both sides of (\ref{s0s1bis}) as functions of the variable $t$ algebraic over $\Q(t)$, the above equality shows that they coincide, and we shall now consider these algebraic functions near $0$ (we choose an arbitrary branch). Considering the Taylor expansions at $0$ in \eqref{s0s1bis}, we obtain: 
$$
\frac{-d_{1,0} - \Delta_1}{\sqrt{d_{1,0}^{2}-{\Delta_1}^2}}\frac{\sqrt{d_{0,1}^{2}-{\Delta_2}^2}}{-d_{0,1} - \Delta_2} + O(t)= \left(\frac{1}{t^2}\left(\frac{1}{d_{1,-1}d_{-1,1}} + O(t)\right)\right)^{\ell+\frac12}
$$
where $\Delta_1$ and $\Delta_2$ are suitable square roots of   
$
d_{1,0}^{2}-4d_{1,1}d_{1,-1}
$
and 
$ 
d_{0,1}^{2}-4d_{1,1}d_{-1,1}
$ 
respectively, and none of the numbers $-d_{1,0} - \Delta_1,\sqrt{d_{1,0}^{2}-{\Delta_1}^2}, \sqrt{d_{0,1}^{2}-{\Delta_2}^2}, -d_{0,1} - {\Delta_2}^2$ is zero because $d_{1,1} \neq 0$. 
This equality is impossible.
\end{proof}

\subsection{Proof of Theorem \ref{theo:maintheotransc}} 
 We shall use the criteria of Proposition~\ref{prop:criteria}  applied to ${\mathbf{b}_{2}}$. From the expression of $a_{3}$ and $a_{4}$ given in Section \ref{sec1}, we may deduce that $a_{3}\neq a_{4}$ and therefore $\Delta^x_{[x_0:x_1]}$ seen as a function on $\P1(\C)$ has at most a simple zero at $P_{1}$ and $P_{2}$. With \eqref{equ:secondmembreform} we find that $P_{1}$ and $P_{2}$ are poles of ${\mathbf{b}_{2}}$.

If  $d_{1,1}=d_{1,0}=0$ (and $d_{0,1}\neq 0$ by Assumption \ref{assumption:oneofthefive}), then a direct calculation shows that 
the polar divisor of ${\mathbf{b}_{2}}$\footnote{{We recall that the polar divisor of $\mathbf{b}_{2}$ is the formal $\Z$-linear combination of points of $\Etproj$ given by $\sum_{P \in \Etproj} n_{P} P$ where $n_{P}$ is equal to $0$ if $P$ is not a pole of $\mathbf{b}_{2}$ and equal to the order of $P$ as a pole of $\mathbf{b}_{2}$ otherwise.}} on $\Etproj$ is 
$
3P_{1}+Q_{2}+\iota_{1}(Q_{2})
$ 
where 
\begin{itemize}
\item $P_{1}=P_{2}=Q_{1}=([1:0],[1:0])$,
\item $Q_{2}= ([-d_{-1, 1}:d_{0, 1}],[1:0])$,
\item $\iota_{1}(Q_{2})=([-d_{-1, 1}:d_{0, 1}], [-td_{1,-1}d_{-1,1}:d_{0,1}(1-td_{0,0})]) \neq Q_{2}$.
\end{itemize}
 The result is now a direct consequence of Proposition \ref{prop:criteria} because $P_{1}$ is a pole of order three of ${\mathbf{b}_{2}}$, and all the other poles of ${\mathbf{b}_{2}}$ have order $1$. 
 
 The case $d_{1,1}=d_{0,1}=0$ is similar.  \par 

{
Assume that  $d_{1,1}=0$ and $d_{1,0}d_{0,1}\neq 0$. In this case, we have 
\begin{itemize}
\item $P_{1}=Q_{1}=([1:0],[1:0])$,
\item  $P_{2}=\iota_{1}(Q_{1})=([1:0],[-d_{1,-1}:d_{1,0}])$,
\item $Q_2=([-d_{-1, 1}:d_{0, 1}],[1:0])$,
\item $\iota_1(Q_2)=([-d_{-1, 1}:d_{0, 1}], [-td_{1,-1}d_{-1,1}:d_{0,1}(1-td_{0,0})+td_{1,0}d_{-1,1}])$.
\end{itemize}
Note that these four points are two by two distinct (since
$d_{0,1} \neq0$ and $t$ is transcendental,  the quantity $d_{0,1}(1-td_{0,0})+td_{1,0}d_{-1,1}$ does not vanish).
A direct computation shows that  the polar divisor of ${\mathbf{b}_{2}}$ on $\Etproj$ is 
$
2P_{1}+2 P_2+Q_{2}+\iota_{1}(Q_{2})
.$ Proposition \ref{prop:quotient1} ensures that $P_{1} \not \sim P_{2}$. So, $P=P_{1}$ or $P_{2}$ is such that none of the $\sigma^i(P)$ with $i \in \Z \backslash \{0\}$ is a pole of order $\geq 2$ of ${\mathbf{b}_{2}}$. The result is now a consequence of Proposition \ref{prop:criteria}. }

Last, assume that $d_{1,1}\neq 0$.
Then, combining Proposition \ref{prop:quotient1} and Proposition \ref{prop:quotient2}, and using the fact that the set of poles of ${\mathbf{b}_{2}}$ is included in $\{P_1,P_2,Q_1,Q_2,\sigma^{-1}(Q_1),\sigma^{-1}(Q_2)\}$, we get that $P_{1}$ is such that none of the $\sigma^i(P_{1})$ with $i \in \Z \backslash \{0\}$ is a pole of ${\mathbf{b}_{2}}$. The result is now a consequence of Proposition \ref{prop:criteria}.

\bibliography{walkbib}
\end{document}